\documentclass[11pt]{article}

\usepackage[usenames,dvipsnames]{color}

\usepackage{amssymb, mathrsfs}
\usepackage{amsthm}
\usepackage{amsmath,color}
\usepackage{graphicx, tikz-cd}
\usepackage{yfonts}
\usepackage{amsbsy}
\usepackage{hyperref}
\usepackage{multicol}
\usepackage[margin=1in]{geometry}
\usepackage{listings}
\usepackage{comment} 
\usepackage[utf8]{inputenc}
\usepackage{float}
\usepackage{breqn}
\usepackage{cleveref}

\makeatletter
\newcommand{\address}[1]{\gdef\@address{#1}}
\newcommand{\email}[1]{\gdef\@email{\url{#1}}}
\newcommand{\@endstuff}{\par\vspace{\baselineskip}\noindent\small
\begin{tabular}{@{}l}\scshape\@address\\\textit{E-mail address:} \@email\end{tabular}}
\AtEndDocument{\@endstuff}
\makeatother
\address{Department of Mathematics, Rice University}
\email{yandi.wu@rice.edu}

\theoremstyle{plain}
\newtheorem{theorem}{Theorem}[section]

\newtheorem{lemma}[theorem]{Lemma}
\newtheorem{question}[theorem]{Question}
\newtheorem{proposition}[theorem]{Proposition}

\theoremstyle{definition}

\newtheorem{definition}[theorem]{Definition}
\newtheorem{remark}[theorem]{Remark}
\newtheorem{construction}[theorem]{Construction}
\newtheorem{algorithm}[theorem]{Algorithm}

\allowdisplaybreaks
\begin{document}

\title{Iso-length-spectral Hyperbolic Surface Amalgams} 
\date{}
\author{Yandi Wu} 

\maketitle 

\begin{abstract}
    Two negatively curved metric spaces are \textit{iso-length-spectral} if they have the same multisets of lengths of closed geodesics. A well-known paper by Sunada provides a systematic way of constructing iso-length-spectral surfaces that are not isometric. In this paper, we construct examples of iso-length-spectral \textit{surface amalgams} that are not isometric, generalizing Buser's combinatorial construction of Sunada's surfaces. We find both homeomorphic and non-homeomorphic pairs. Finally, we construct a noncommensurable pair with the same \textit{weak length spectrum}, the length set without multiplicity.
\end{abstract}

\section{Introduction}

Recall that the \textit{(unmarked) length spectrum} of a metric space $(X, g)$ is the ascending multiset of positive real numbers representing lengths of closed geodesics in $(X, g)$. We say that two metric spaces are \emph{iso-length-spectral} if they have the same length spectra.  

The question of whether there exist iso-length-spectral manifolds that are not isometric has a long history. The first examples of such objects are due to Milnor (see \cite{milnor}), who constructed two flat, 16-dimensional non-isometric, iso-length-spectral tori. In \cite{vigneras}, Vigner\'as constructed examples of hyperbolic surfaces that are iso-length-spectral but not isometric. Both Vigner\'as's and Milnor's constructions are heavily number theoretic and, one could argue, difficult to describe geometrically. In \cite{sunada}, Sunada developed a celebrated, systematic, and more geometric way of constructing hyperbolic iso-length-spectral surfaces. In \cite{buser1} (and \cite{buser2} for the genus $6$ case), Buser uses \cite{sunada} to construct pairs of iso-length-spectral, non-isometric surfaces of genus $\geq 5$. 

Due to the combinatorial nature of Buser's approach, one can apply his construction to metric spaces outside the setting of Riemannian manifolds. We explore such objects in this paper, and show that in contrast to the hyperbolic surface case, Buser's interpretation of the Sunada construction sometimes yields examples of \textit{non-homeomorphic} iso-length-spectral objects. 

Let $(X, g)$ be a simple, thick hyperbolic surface amalgam, which, roughly speaking, is constructed by isometrically gluing together compact, hyperbolic surfaces with boundary together along their boundary components. We refer the reader to Definition 2.3 of \cite{lafont} for a more precise definition; note that in his paper, he refers to surface amalgams as ``2-dimensional P-manifolds". We now state our first result: 

\begin{theorem}\label{main}
    There exist iso-length-spectral surface amalgams equipped with piecewise hyperbolic metrics that are \emph{not} isometric. Furthermore, there exist non-homeomorphic iso-length-spectral pairs of hyperbolic surface amalgams. 
\end{theorem}

We remark that in contrast to the surface amalgam case, there cannot exist non-homeomorphic iso-length-spectral pairs of hyperbolic closed surfaces, as the length spectrum completely determines the genus of a surface (see \cite{chavel}). In the setting of $3$-manifolds, on the other hand, applications of the Sunada method (e.g. \cite{spatzier}, \cite{reid}, \cite{mcreynolds}) have been shown to yield hyperbolic, iso-length-spectral 3-manifolds which are non-isometric and thus non-homeomorphic due to Mostow Rigidity. This provides further evidence towards the fact that surface amalgams share characteristics of both surfaces and 3-manifolds (see also \cite{hst}). 

We say that two metric spaces are \textit{metrically commensurable} if they share some isometric finite-sheeted cover. Notably, both main sources of iso-length-spectral, non-isometric hyperbolic manifolds (from \cite{vigneras} and \cite{sunada}) yield metrically commensurable manifolds. This led Reid to pose an interesting open question in \cite{reid}: 

\begin{question}[Reid]\label{question:reid} Do there exist two iso-length-spectral hyperbolic manifolds which are not metrically commensurable?
\end{question}

It is known that all compact iso-length-spectral \textit{arithmetic} surfaces and 3-manifolds are necessarily metrically commensurable (see \cite{reid} and \cite{chlr}). In \cite{lsv}, in contrast, the authors find large families of locally symmetric, iso-length-spectral manifolds of higher rank that are not metrically commensurable. However, the question of whether there exist iso-length-spectral, metrically non-commensurable hyperbolic surfaces is still open. 

We say two topological spaces are \textit{topologically commensurable} if they share \textit{homeomorphic} (as opposed to isometric) finite-sheeted covers. Note that if two metric spaces are not topologically commensurable, they are automatically not metrically commensurable. For the remainder of the paper, when we say ``commensurable," we mean topologically commensurable. 

We follow the terminology from \cite{GR} and define the \textit{weak length spectrum} of a locally CAT($-1$) metric space to be a collection of lengths of closed geodesics \textit{without} multiplicity. If two metric spaces have the same weak length spectrum, we say they are \textit{weak length isospectral}. Notice that since the weak length spectrum is a subset of the length spectrum, weak length isospectrality is a weaker condition than length isospectrality. 

In fact, in \cite{lmnr}, the authors show that examples of weak length isospectral (called ``length equivalent" in their paper) hyperbolic manifolds that are not iso-length-spectral exist in great abundance. The examples they construct arise from sequences of manifolds in towers of covers so that the weak length isospectral pairs have different volumes. In fact, the ratios between volumes of consecutive terms $M_{n + 1}$ and $M_n$ in the sequences tend to infinity.  

We are now ready to state the second main result of the paper:  

\begin{theorem}\label{theorem:commensurable}
    There exist weak length isospectral surface amalgams equipped with piecewise hyperbolic metrics that are \emph{not} (topologically) commensurable. 
\end{theorem}

We remark that it is impossible to find pairs of surfaces that are not topologically commensurable, as all closed surfaces are commensurable as topological objects. The proof of \Cref{theorem:commensurable} relies on work from \cite{stark} and \cite{DST} which is related to the abstract commensurability classification problem of right-angled Coxeter groups. Furthermore, in contrast to the examples from \cite{lmnr}, the weak length isospectral pairs from Theorem \ref{theorem:commensurable} have the same volume. \\

\noindent \textbf{Outline of the paper.} We now give a brief outline of the paper. We begin with a brief review of Buser's techniques for constructing iso-length-spectral, non-isometric surfaces in Section 2. We continue with a construction of pairs of homeomorphic, iso-length-spectral, non-isometric surface amalgams, followed by a construction of a pair that is not homeomorphic using Lafont's criteria from \cite{lafont} in Section 3. Instead of using Buser's transplantation technique, we count copies of ``identical" closed geodesics. Finally, in Section 4, we construct a pair of weak length isospectral surfaces which are not commensurable using criteria from \cite{DST}. This time, we prove weak length isospectrality using Buser's transplantation technique.\\

\noindent
\textbf{Acknowledgements.} I am deeply indebted to Chris Leininger for suggesting the project, discussing it in depth with me, and pointing out gaps in previous drafts. I also thank Alan Reid for discussions that inspired \Cref{theorem:commensurable}, and for pointing out various helpful references. Finally, I thank my doctoral advisors Tullia Dymarz and Caglar Uyanik for helpful discussions and moral support. I also thank the referee for a careful reading of the paper, and for suggestions on improving exposition. This work was partially supported by an NSF RTG grant DMS-2230900.

\section{Buser's Techniques}
\label{Buser}
We rely heavily on the techniques from \cite{buser1}, which we briefly sketch.

\subsection{Buser's original construction}
Buser constructs two iso-length-spectral but non-isometric hyperbolic genus 5 surfaces by gluing together identical copies of right-angled octagons, which he calls \textit{building blocks}. The gluing scheme is shown below in \Cref{fig:s1}. His construction is modeled off Sunada's construction using almost conjugate subgroups of $(\mathbb{Z}/8\mathbb{Z})^{\times} \ltimes (\mathbb{Z}/8\mathbb{Z})^{+}$ given by Gerst in \cite{gerst} (see Example 1 of Section 1 of \cite{sunada}). 

\begin{figure}[H]
    \centering
    \includegraphics[width=\textwidth]{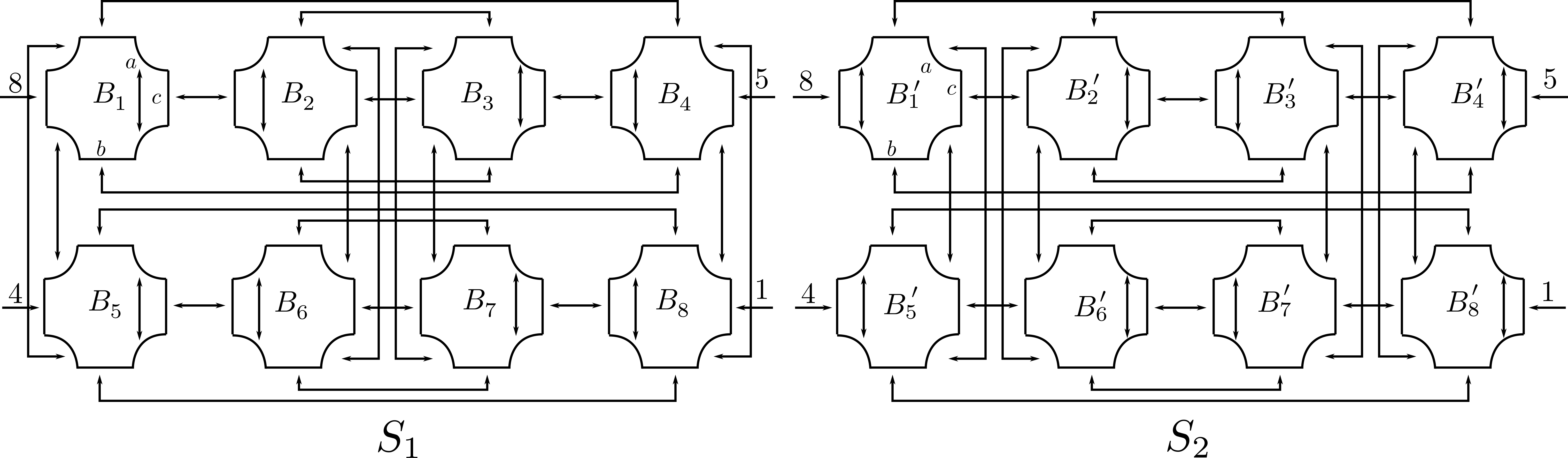}
    \caption{Gluing schemes for iso-length-spectral, non-isometric a genus 5 surfaces from \cite{buser1}.}
    \label{fig:s1}
\end{figure}

We now briefly summarize the idea behind showing Buser's genus $5$ surfaces are iso-length-spectral but not isometric. We do not detail the higher genus cases, but the ideas are similar with variations in how the building blocks are constructed. \\

\noindent \textbf{Buser's surfaces are not isometric.} The building blocks $B_i$ ($1 \leq i \leq 8$) have three sets of identical edges with lengths $a$, $b$ and $c$ (see \Cref{fig:s1}). By adding an extra restriction on the lengths $b$ and $c$, Buser is able to conclude the systoles of $S_1$ and $S_2$ are exactly geodesics of length $c$:

\begin{lemma}[Lemma 3.3, Proposition 3.4 of \cite{buser1}] Let $0 < c < b < 1$. Then any geodesic curve $\delta$ of $B_i$ which connects two sides of $B_i$ has length $\ell(\delta) \geq c$. Equality holds only if $\delta$ is a side (of length $c$) of a building block. 
\end{lemma}

As a result, $S_1$ and $S_2$ each have sets of four systoles $\{\gamma_i\}$ and $\{\gamma'_i\}$ ($1 \leq i \leq 4$) respectively of length $c$. In $S_1$, the $\gamma_i$ are between $B_i$ and $B_{i + 1}$ for $i \in \{1, 3, 5, 7\}$. On the other hand, in $S_2$, the $\gamma'_i$ are between $B'_i$ and $B'_{i + 1}$ for $i \in \{2, 4, 6, 8\}$ (where, as always, $i + 1$ is taken mod 8). One can easily check that cutting along the multicurve $\{\gamma_i\}$ in $S_1$ yields a single connected component that is topologically a torus with 8 boundary components, while cutting along $\{\gamma'_i\} \subset S_2$ yields two connected components, each of which is topologically a torus with four boundary components. If $S_1$ and $S_2$ were isometric, cutting along their systoles would yield homeomorphic connected components, which is clearly not the case here. \\

\noindent \textbf{Buser's surfaces are iso-length-spectral.} Next, for every closed geodesic in $S_1$, Buser constructs one in $S_2$ with the same length (and vice versa) using a technique he calls \textit{transplantation}. Given a closed geodesic $\gamma \subset S_1$ with a starting point $p$ in the interior of some building block $B_i \subset S_1$, Buser specifies which building block $B'_{k(i)} \subset S_2$ one should start constructing $\gamma' \subset S_2$ in so that $\ell(\gamma) = \ell(\gamma')$. He uses the following set of rules, which depend on the parities of $\#a$ and $\#b$, defined below:

\begin{algorithm}[5.3 (Initiation), \cite{buser1}]
\label{algorithm:originit} Let $\#a$ and $\#b$ be the number of times a curve $\gamma$ transversely crosses sides of lengths $a$ and $b$ respectively, and let $B_{n_1}$ (resp. $B'_{n'_1}$) denote the building blocks in $S_1$ (resp. $S_2$) in which to initiate $\gamma$ (resp. $\gamma'$). 

\begin{enumerate}
\item If $\#a$ is even, take $n_1 = n'_1$;
\item If $\#a$ is odd and $\#b$ is even, take $n'_1 = n_1 + 1$;
\item If $\#a$ and $\#b$ are both odd, take $n'_1 = n_1 + 2$. 
\end{enumerate}
\end{algorithm}

To construct $\gamma'$, one simply decomposes $\gamma$ into geodesic segments $\{\gamma_j\}_{j = 1}^{N}$ such that each $\gamma_j$ is contained completely in some building block $B_j$ and $\gamma_1$ and $\gamma_N$ are contained in the same building block and share an endpoint at $p$. Then, for each $\gamma_j$, taking advantage of the fact that the building blocks in $S_1$ and $S_2$ are all identical, one can construct a $\gamma'_j \subset B'_j$ identical to each $\gamma_j \subset B_j$. Buser shows that following \Cref{algorithm:originit}, the set of $\{\gamma'_j\}$ will always close up to a geodesic loop. Furthermore, since $\ell(\gamma_j) = \ell(\gamma'_j)$ for all $j$, it follows that $\ell(\gamma) = \ell(\gamma')$, as desired. The same set of initiation rules also applies in the other direction (constructing a geodesic in $S_1$ given one in $S_2$).

\subsection{Buser's techniques in the surface amalgam setting} We now specify some notation and establish some facts used in the iso-length-spectrality proofs in the remainder of the paper. The results and definitions in this section are either heavily inspired by or taken directly from \cite{buser1}. 

Let $\beta_i$ (resp. $\beta'_i$) be a connected geodesic segment in $S_1$ (resp. $S_2$) which can be written as a union $\bigcup\limits_{k = 1}^L \gamma_{i, k}$ (resp. $\bigcup\limits_{k = 1}^L \gamma'_{i, k}$ ) of segments each completely contained in a single building block. We will see later that every closed geodesic in a surface amalgam constructed in this paper can be written as a concatenation of $\beta_i$'s or $\beta'_i$'s. We define \begin{equation}\label{eqn:delta}\delta_i(k) = n'_{i, k} - n_{i, k} \pmod 8,\end{equation} where $n_{i, k}$ (resp. $n'_{i, k}$) is the index of the building block containing $\gamma_{i, k}$ (resp. $\gamma'_{i, k}$). Thus, when we say ``initiate $\beta'_i$ with the rule $\delta_i(0) = N$," we mean that we will set $n'_{i, 0}$ equal to $n_{i, 0} + N$. In other words, if $\beta_i$ starts in $B_{n_{i, 0}} \subset S_1$, then $\beta'_i$ will start in $B'_{n_{i, 0} + N} \subset S_2$. 

For the convenience of the reader, we list observations from Section 5 of \cite{buser1} used to prove the validity of \Cref{algorithm:originit}, which may be checked. We will also use these observations extensively in iso-length-spectrality proofs in the remainder of the paper. Note that while Buser works with closed surfaces, the segments $\beta_i$ and $\beta'_i$ are contained entirely within closed surfaces, so the observations are still applicable in the setting of surface amalgams. 

\begin{lemma}[c.f. Proof of 5.3 in \cite{buser1}] \label{lemma:delta} Let $\delta_i$, $\beta_i$, and $\beta'_i$ be as above. Then the following changes to $\delta_i$ are observed whenever $\beta_i$ and $\beta'_i$ cross edges of building blocks of $S_1$ and $S_2$ respectively: 
\begin{enumerate}
    \item \emph{Crossing a side of length $a$.} If $\delta_i(k)$ is even, then $\delta_i(k + 1) = \delta_i(k) + 4 \pmod 8$. If $\delta_i(k)$ is odd, then $\delta_i(k + 1) = \delta_i(k)$. 
    \item \emph{Crossing a side of length $b$.} If $\delta_i(k) = 0$ or $4$, $\delta_i(k + 1) = \delta_i(k)$. If $\delta_i(k) = \pm 2$, $\delta_i(k + 1) = \delta_i(k) + 4 \pmod 8$. There are other possible scenarios, but only these are used in this paper. 
    \item \emph{Crossing a side of length $c$.} Regardless of the value of $\delta_i(k)$, $\delta_i(k + 1) = \delta_i(k)$. 
\end{enumerate}
\end{lemma}

\subsubsection{Translated and transplanted copies} Next, we define \textit{translated copies} of curves, which, roughly speaking, are locally isometric copies of a curve on the \textit{same} surface (either $S_1$ or $S_2$). In contrast, a \textit{transplanted copy} of a curve in $S_1$ is a curve on the \textit{other} surface $S_2$ which has the same length (and vice versa). To formalize this, we first present a definition from \cite{buser1}:

\begin{definition} [c.f. Definition 5.2, \cite{buser1}] \label{def:locallycong} Let $\beta$ and $\beta'$ be two curves in $S_1$ or $S_2$ which have the same length. Suppose $\beta = \bigcup\limits_{k = 1}^{L} \beta_k$, where each $\beta_k$ is a curve contained entirely in a single building block. Similarly, suppose $\beta'$ can be written as the union $\beta' = \bigcup\limits_{k = 1}^{L} \beta'_k$. Then $\beta$ and $\beta'$ are \textit{locally congruent} if there exist local isometries $\{\varphi_k\}_{k = 1}^L$ such that $\varphi_k(N(\beta_k)) = N(\beta'_k)$ for every $k \in [1, L]$, where $N(\beta_k)$ and $N(\beta'_k)$ are neighborhoods of $\beta_k$ and $\beta'_k$ respectively. 
\end{definition} 

We can then define translated copies:

\begin{definition} [Translated copies of geodesic segments] Given a geodesic segment $\beta = \bigcup\limits_{k = 1}^L \beta_k$ in $S_1$, a \textit{translated copy} of $\beta$ is a curve $\alpha = \bigcup\limits_{k = 1}^L \alpha_k \subset S_1$ which is locally congruent to $\beta$. We also require that if $\beta$ begins (resp. ends) on an edge with a certain label, then $\alpha$ begins (resp. ends) on an edge with the same label. Furthermore, if $\beta_k$ and $\alpha_k$ are contained in the building blocks $B_{n(k)}$ and $B_{m(k)}$ respectively, then for every $1 \leq k \leq L - 1$, $B_{m(k)}$ meets $B_{m(k + 1)}$ along an edge with the same label as that of the edge where $B_{n(k)}$ and $B_{n(k + 1)}$ meet. One can also replace $S_1$ in the definition with $S_2$.
% Suppose $\beta^1_i := \beta_i \subset S_1$ begins at a point $p_{i, 1} \in \gamma_{i, 1} \subset B_{n(i, 1)}$, where $\gamma_{i, 1}$ is a geodesic segment  contained in some building block $B_{n(i, 1)}$. Further suppose that $\gamma_{i, 1}$ is a subset of the gluing geodesic of $X$ which is identified with geodesic segments in $M$ other building blocks. A \textit{translated copy $\beta^j_i$} of $\beta_i$ ($j = 1, 2, ..., M$) is a connected segment in $S_1$ that is locally congruent to $\beta_{i, 1}$ and  begins at a point $p_{i, j} \in \gamma_{i, j} \subset B_{n(i, j)}$ (where as before, the building block $B_{n(i, j)}$ contains the geodesic segment $\gamma_{i, j}$) such that $p_{i, j}$ is identified with $p_{i, 1}$ when $\gamma_{i, j}$ is identified with $\gamma_{i, 1}$ to create a segment of the gluing geodesic.
\end{definition} 

Following terminology from Section 11.6 of \cite{buser2}, we also define the following. 

\begin{definition}[Transplanted copies of geodesic segments] Given a geodesic segment $\beta = \bigcup\limits_{k = 1}^L \beta_k$ in $S_1$, a \textit{transplanted copy} of $\beta$ is a curve $\beta' = \bigcup\limits_{k = 1}^L \beta'_k \subset S_2$ which is locally congruent to $\beta$. We also require that if $\beta$ begins (resp. ends) on an edge with a certain label, then $\beta'$ begins (resp. ends) on an edge with the same label. Furthermore, if $\beta_k$ and $\beta'_k$ are contained in the building blocks $B_{n(k)} \subset S_1$ and $B'_{n(k)} \subset S_2$ respectively, then for every $1 \leq k \leq L - 1$, $B'_{n(k)}$ meets $B'_{n(k + 1)}$ along an edge with the same label as that of the edge where $B_{n(k)}$ and $B_{n(k + 1)}$ meet. 
\end{definition}

\section{Proof of Theorem \ref{main}}

We now proceed with our proofs of the existence of iso-length-spectral, non-isometric examples. 

\subsection{Homeomorphic, non-isometric surface amalgams} 
\label{subsection:homeo}
We first construct two homeomorphic, iso-length-spectral, non-isometric surface amalgams. 

\begin{construction}\label{homeomorphic}
Consider $S_1$ and $S_2$ from \cite{buser1}. In each surface, consider the two closed geodesics of length $2a$ obtained from concatenating the top right edge of length $a$ on $B_2$ (resp. $B_6$) with the top left edge of length $a$ on $B_3$ (resp. $B_7$). We do the same for the bottom edges. We then identify all the red edges of length $a$ on $B_2$ and $B_6$ (resp. $B'_2$ and $B'_6$) from \Cref{fig:S1S2homeo} and all blue edges on $B_3$ and $B_7$ (resp. $B'_3$ and $B'_7$).  
\end{construction}

\begin{figure}[h!]
    \centering
\centerline{\includegraphics[width=1.1\textwidth]{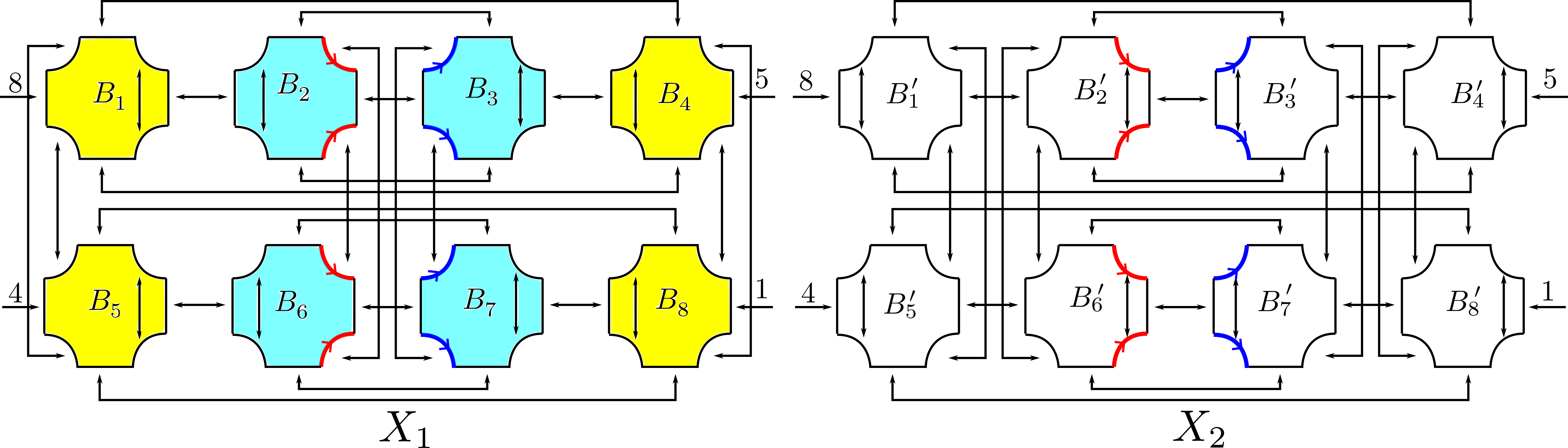}}
    \caption{Homeomorphic, non-isometric, iso-length-spectral surface amalgams. The gluing curve consists of two closed geodesics of length $2a$ identified together. Orientations of the gluing curves are specified in the figure. Geodesic segments that are the same colors in $X_1$ are identified; the same is true for $X_2$. Cutting along the gluing geodesics and the systoles yields two connected components (in yellow and blue) for $X_1$ but only one connected component for $X_2$.}
    \label{fig:S1S2homeo}
\end{figure}

\begin{proposition}\label{nonhomeoproof} The surface amalgams $X_1$ and $X_2$ from \Cref{homeomorphic} are iso-length-spectral and homeomorphic, but not isometric. 
\end{proposition}

\begin{proof} 
\noindent \textbf{$X_1$ and $X_2$ are homeomorphic.} Note that cutting along the four geodesics that are identified will yield a genus $3$ surface with four boundary components of length $2a$ for both $S_1$ and $S_2$. The two loops of length $2a$ are nonseparating for both $S_1$ and $S_2$ so cutting along them yields homeomorphic surfaces $S_{3, 4}$. Since there is bijection between homeomorphic chambers of $X_1$ and $X_2$, they are homeomorphic due to \cite{lafont}. \\

\begin{figure}[H]
    \centering
    \includegraphics[width=0.5\textwidth]{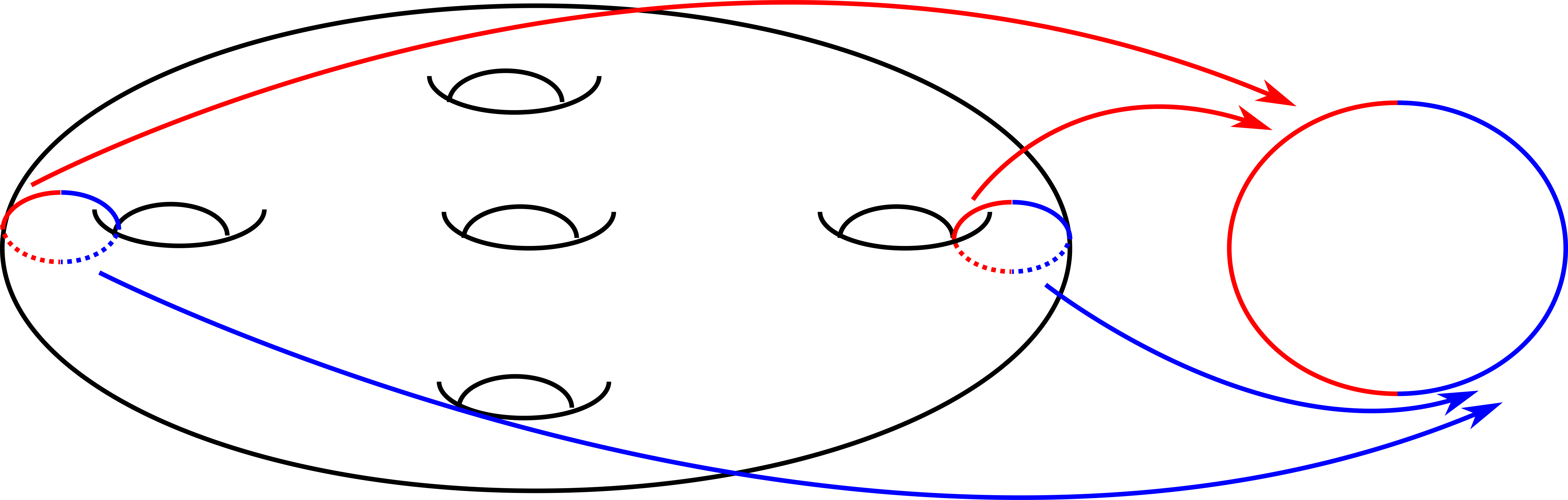}
    \caption{$X_1$ and $X_2$ are both homeomorphic to the above surface amalgam, created from identifying two non-separating closed geodesics on Buser's genus 5 example.}
    \label{fig:homeo}
\end{figure}

\noindent \textbf{$X_1$ and $X_2$ are not isometric.} In order for two surface amalgams to be isometric, there must be a bijection between isometric chambers due to the bijective correspondence between homeomorphic chambers (see \cite{lafont}). As with the examples in \cite{buser1}, the only systoles in $C_1$ and $C_2$, chambers in $X_1$ and $X_2$ respectively, are the four closed geodesics of length $c$. Cutting along these systoles yields two connected components for $C_1$ and one connected component for $C_2$; thus, $C_1$ and $C_2$ cannot be isometric.\\ 

\noindent \textbf{$X_1$ and $X_2$ are iso-length-spectral.} It suffices to consider geodesics that do not intersect the gluing curve, as \cite{buser1} already shows a 1-1 correspondence between closed geodesics completely contained in $S_1$ and $S_2$. Consider $\gamma$, a closed geodesic in $X_1$. Decompose $\gamma$ into $\bigcup\limits_{i = 1}^N \beta_i$ so that each $\beta_i$ is a continuous geodesic segment in the original surface $S_1$ with endpoints on the gluing curve. %Furthermore, we require that none of the interiors of the $\beta_i$'s intersect the gluing curve. 
We say a closed geodesic in $\gamma_1 \subset X_1$ is \textit{identical} to $\gamma$ if $\gamma_1$ and $\gamma$ begin at the same point $x \in X_1$ on the gluing curve, and $\gamma_1$ can be decomposed into $\bigcup\limits_{i = 1}^N \beta_i^1$, a union of translated copies of $\beta_i$ in $S_1$ with endpoints on the gluing curve. There is a natural isometry $\varphi$ between gluing geodesics in $X_1$ and those in $X_2$. We say $\gamma' \subset X_2$ is \textit{identical} to $\gamma \subset X_1$ if $\gamma'$ begins and ends at $\varphi(x) \in X_2$ while $\gamma$ begins and ends at $x \in X_1$, and $\gamma_2$ can be decomposed into $\bigcup\limits_{i = 1}^N \beta'_i$ so that each $\beta'_i$ is a transplanted copy of $\beta_i$ (in the sense of \cite{buser1}) that begins and ends on the gluing curve in $X_2$. We will argue that the numbers of closed geodesics identical to $\gamma$ in $X_1$ and $X_2$ are the same, which shows the length spectra are the same. 

We now construct closed geodesics identical to $\gamma$ in $X_1$ by concatenating \textit{admissible} translated copies of each $\beta_i$, which we define precisely below: 

\begin{definition}\label{def:admissible} We say that a choice of geodesic segment $\beta_i$, where $2 \leq i \leq N$, is \textit{admissible} if it is compatible with the previous choices of $\beta_j$ ($j < i$) in the following sense: 
\begin{enumerate}
    \item \textit{$\beta_{i - 1} \cup \beta_i$ does not backtrack}: $\beta_i$ does not start on the same side of the building block that $\beta_{i - 1}$ ends on (see (a) on \Cref{fig:inadmissible}), \textit{and} if $i = N$, $\beta_i$ does not end on the same side of the building block $\beta_1$ starts on; 
    \item \textit{$\beta_{i - 1} \cup \beta_i$ is continuous}: $\beta_i$ begins at the same point on the gluing curve that $\beta_{i - 1}$ ends on (see (b) on \Cref{fig:inadmissible}), \textit{and}, if $i = N$, $\beta_i$ ends at the same point on the gluing curve that $\beta_1$ begins at. 
\end{enumerate}
\end{definition} 

\begin{figure}[H]
    \centering
    \includegraphics[width=0.5\textwidth]{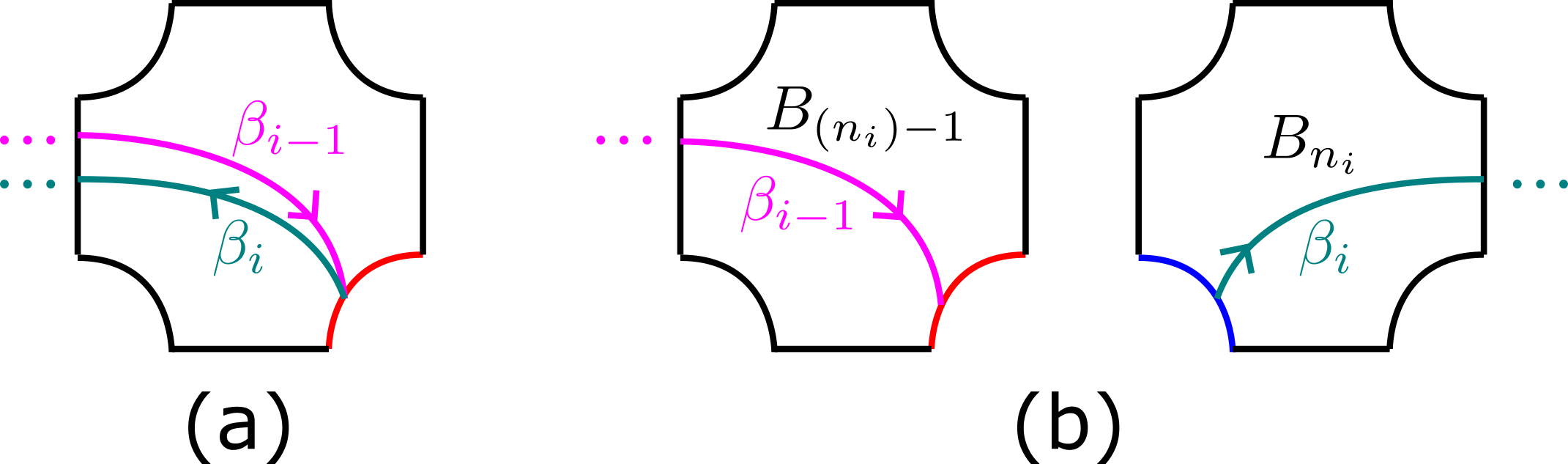}
    \caption{Two examples of (parts of) inadmissible $\beta_i$'s. In (a), $\beta_{i - 1} \cup \beta_i$ is a backtracking geodesic segment. In (b), $\beta_{i - 1}$ and $\beta_i$ begin at different points of the gluing curve, as the red and blue segments are not identified, which means $\beta_{i - 1} \cup \beta_i$ is disconnected.}
    \label{fig:inadmissible}
\end{figure}

We first prove the following:\\

\noindent \textbf{Claim.} \textit{In $X_1$, there is exactly one translate of $\beta_i^1 := \beta_i$, $\beta_i^2$, which begins and ends at the same point on the gluing curve $\beta_i$ begins and ends at.\\}

\noindent \textit{Proof of claim:} If $\beta_i$ begins on the top right (resp. top left, bottom right, or bottom left) side of length $a$ in $B_i$, initiate $\beta_i^2$ at the same point on the top right (resp. top left, bottom right, or bottom left) side of length $a$ in $B_{i + 4}$. Let $\beta_i^j = \bigcup\limits_{k = 0}^{L} \gamma_{i, k}^j$ be a decomposition of $\beta_i^j$ into geodesic segments that are contained entirely in a single building block, where $\gamma_{i, k}^j \subset B_{n_{i, k}^j}$ and $j = 1, 2$. Let $\epsilon_k = n_{i, k}^2 - n_{i, k}^1.$ Note in particular that $\epsilon_0 = 4$. We observe the following changes to $\epsilon_j$ in the following scenarios:

\begin{itemize}
    \item \textit{$\beta^1_i$ and $\beta^2_i$ cross sides of length $c$.} In this case, $n^j_{i, k + 1} = n^j_{i, k} \pm 1$ for both $j = 1, 2$ depending on whether the left or right side of a building block is crossed. Note that when crossing edges of length $c$, $\beta^1_i$ and $\beta^2_i$ either both cross left sides or both cross right sides. Then $\epsilon_{k + 1} = (n^2_{i, k} \pm 1) - (n^1_{i, k} \pm 1) = n^2_{i, k} - n^1_{i, k} = \epsilon_k$.
    \item \textit{$\beta^1_i$ and $\beta^2_i$ cross sides of length $a$.} In this case, for both $j = 1, 2$, $n^j_{i, k + 1} = n^j_{i, k}$ or $n^j_{i, k} + 4$, depending on the parity of the building block. Provided that $\epsilon_k = 4 \pmod 8$, either $\epsilon_{k + 1} = n^2_{i, k} - n^1_{i, k} = \epsilon_k$ or $\epsilon_{k + 1} = (n^2_{i, k}  + 4) - (n^1_{i, k} + 4) = \epsilon_k.$
    \item \textit{$\beta^1_i$ and $\beta^2_i$ cross sides of length $b$.} In this case, for both $j = 1, 2$, $n^j_{i, k + 1} = n^j_{i, k} \pm 1$ or $n^j_{i, k} \pm 3$, depending on the label of the building block. But again, if $\epsilon_k = 4$, then $\epsilon_{k + 1} = (n^2_{i, k} \pm 1) - (n^1_{i, k} \pm 1) = n^2_{i, k} - n^1_{i, k} = \epsilon_k$ or $\epsilon_{k + 1} = (n^2_{i, k} \pm 3) - (n^1_{i, k} \pm 3) = n^2_{i, k} - n^1_{i, k} = \epsilon_k$. 
\end{itemize}

From this, we conclude that in fact, $\epsilon_L = 4$. That is, if $\beta^1_i$ ends on an edge of the building block $B_{N}$, then $\beta^2_i$ ends on the same point in the corresponding edge of $B_{N + 4}$. Since there are only two edges that are top right, top left, bottom right, or bottom left edges of building blocks which are identified to create a subarc of the gluing geodesic, there cannot be another translate of $\beta^1_i$ sharing beginning and end points with $\beta^1_i$. Thus, there is exactly one such $\beta^2_i$, as claimed.
\qed\\

We now compute $C_1(\gamma)$, the number of closed geodesics in $X_1$ identical to $\gamma$. Fix a copy of $\beta_1$, say $\beta_1^1 \subset X_1$. By construction, $\beta^1_1$ can be concatenated with any of the two copies of $\beta_2$ in $\{\beta_2^j\}_{j = 1}^2$ to create a geodesic segment \textit{unless} some $\beta_2^j$ is chosen so that $\beta_2^j \cup \beta_1^1$ backtracks, in which case there is only one admissible copy of $\beta_2$ which can be concatenated with $\beta_1^1$. Using the same logic for all $2 \leq i \leq N$, we have the following general fact: 
\begin{align*} c_i &:= \#\{\beta^j_i \text{ which can be concatenated with some fixed copy of } \beta_{i - 1}\} \\&= \begin{cases}
    1 \text{ if $\beta_{i - 1}$ ends on a top red edge from \Cref{fig:S1S2homeo} and $\beta_i$ begins on a top red edge,}\\\text{   or the statement is true if ``red" is replaced with ``blue" and/or ``top" is replaced with ``bottom";}\\
    2 \text{ otherwise.}
\end{cases}
\end{align*}

We now examine how to choose a copy of $\beta_N$ from $\{\beta_N^j\}_{j = 1}^{2}$. The set of admissible $\beta_N$ now depends on our choice of $\beta_{N - 1}^j$ and $\beta_1^1$, and specifically on whether the following facts are true: 

\begin{itemize} 
\item \textit{Fact 1:} $\beta^j_{N - 1}$ ends on a top red edge and $\beta_N$ begins on a top red edge (or the statement is true if we replace ``red" with ``blue" and/or ``top" with ``bottom"); 
\item \textit{Fact 2:} $\beta_N$ ends on a top red edge and $\beta_1$ begins on a top red edge (or the statement is true if we replace ``red" with ``blue" and/or ``top" with ``bottom").
\end{itemize} 

We now do some casework, depending on whether Facts 1 and 2 are satisfied. \\

\noindent \textit{Case One}: Neither Fact 1 nor Fact 2 is true. In this case, any choice of pairs of translates of $\beta_1$ and $\beta_N$ is compatible. Thus, each choice of $\beta_1^j$ ($j = 1, 2$) has $\prod\limits_{i = 2}^{N} c_i = \bigg(\prod\limits_{i = 2}^{N - 1} c_i \bigg)(2)$ choices of sequences of admissible translates of $\beta_i$. Then $C_1(\gamma) = 2\bigg(\prod\limits_{i = 2}^{N - 1} c_i\bigg)(2)$. See \Cref{fig:countingcase1}.

\begin{figure}[h]
    \centering
\centerline{\includegraphics[width=1.15\textwidth]{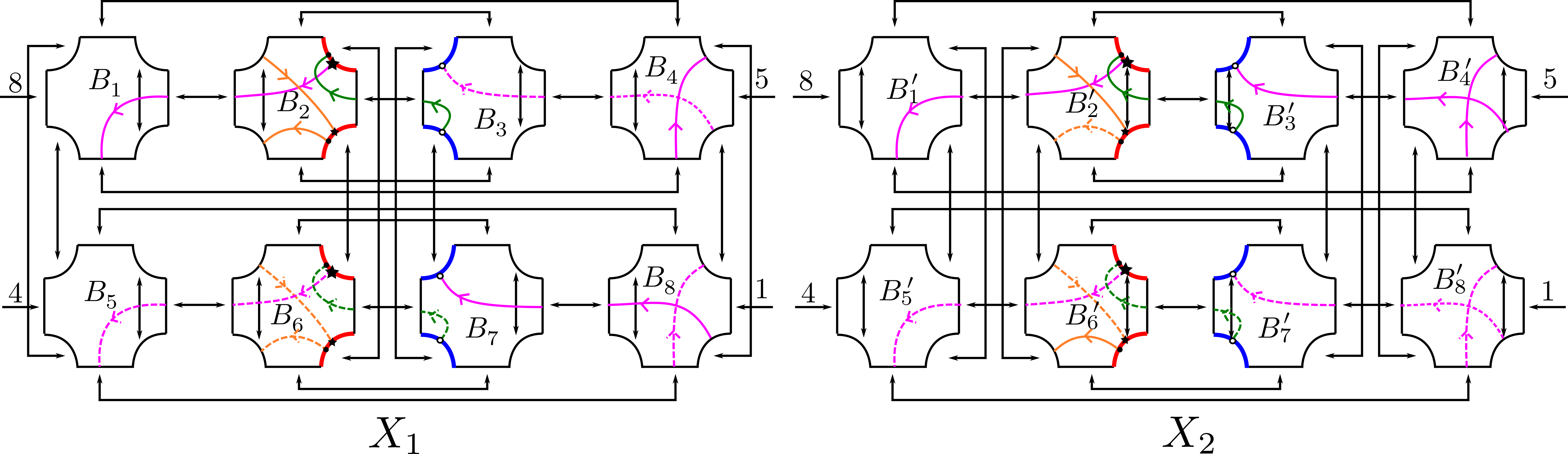}}
    \caption{An illustration of Case 1, with identical copies of closed geodesics created by concatenating the pink ($\beta_1)$, green ($\beta_2$), and orange ($\beta_3$) geodesic segments. Two closed geodesics are illustrated in each of $X_1$ and $X_2$, one solid and one dashed. The star icon indicates the start of each geodesic. Here, $N = 3$ and $c_2 = 2$, so $C_1(\gamma) = 2(2)(2) = C_2(\gamma)$. We can check that there are indeed eight identical, non-backtracking closed geodesics in each of $X_1$ and $X_2$.}
    \label{fig:countingcase1}
\end{figure}

\noindent \textit{Case Two}: One of Fact 1 or Fact 2 is true. Suppose Fact 1 is true. Then any choice of $\beta_N^j$ is compatible with a fixed choice of translate of $\beta_1$, but there is only one admissible translate of $\beta_N$ for each fixed translate of $\beta_{N - 1}$. Suppose Fact 2 is true. Again, given a fixed $\beta_1^j$, there is only one admissible translate of $\beta_N$, as one of them is not compatible with $\beta_1^j$. In either case, we slightly modify the equation for $C_1(\gamma)$ from the previous case: $C_1(\gamma) = 2\bigg(\prod\limits_{i = 2}^{N - 1} c_i\bigg)(1)$. See \Cref{fig:countingcase2}.

\begin{figure}[h]
    \centering
\centerline{\includegraphics[width=1.1\textwidth]{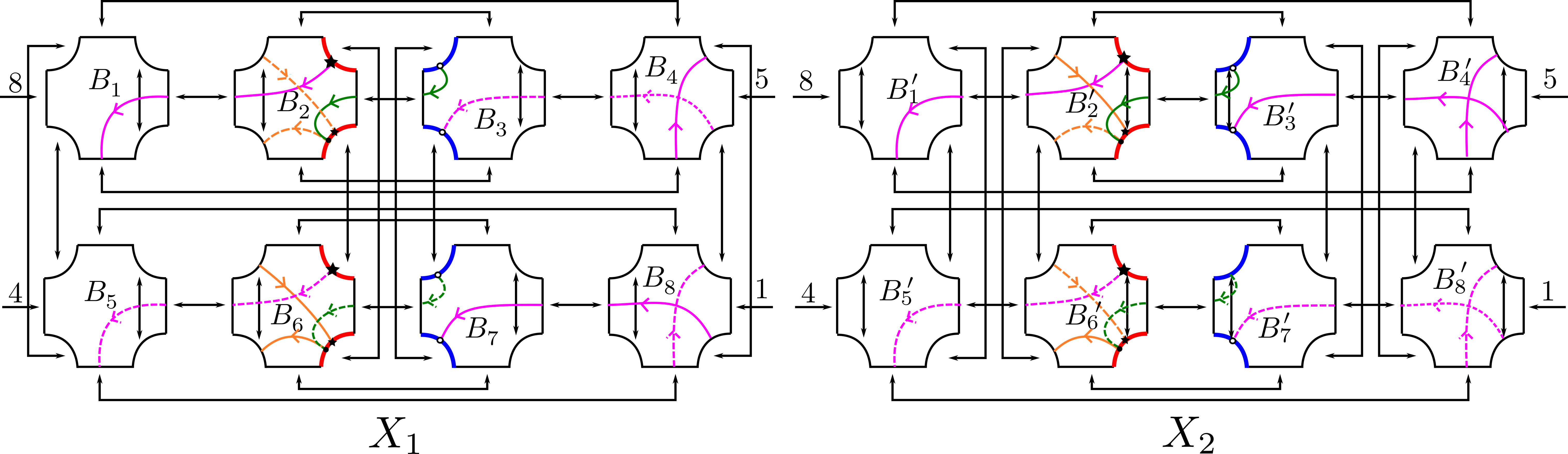}}
    \caption{An illustration for Case 2. Again, two closed geodesics are shown, one solid and one dashed, and the star indicates the starting point of each geodesic. Here, $N = 3$ and $c_2 = 2$, so $C_1(\gamma) = 2(2) = C_2(\gamma)$. We can check that there are four identical, non-backtracking closed geodesics in each of $X_1$ and $X_2$.}
    \label{fig:countingcase2}
\end{figure}

\noindent \textit{Case Three}: Both Fact 1 and Fact 2 are true. This case is the most nuanced, but surprisingly, the count is the same as in Case 2. There are four possible pairs of geodesic segments $\beta_1^j$ and $\beta_{N - 1}^k$ ($j, k \in \{1, 2\}$) which satisfy (2) of \Cref{def:admissible}. For exactly two of these pairs, there is one inadmissible copy of $\beta_N$: the translate of $\beta_N$ with starting point coinciding with the endpoint of $\beta_{N - 1}^k$ and endpoint coinciding with the starting point of $\beta_1^j$. For all other pairs, there are two inadmissible (and thus no admissible) copies of $\beta_N$: one translate with endpoint coinciding with the starting point of $\beta_1^j$ and one (other) translate with starting point coinciding with the endpoint of $\beta_{N - 1}^k$. Thus, $C_1(\gamma) = 2\bigg(\prod\limits_{i = 2}^{N - 1} c_i\bigg)(1) + 2\bigg(\prod\limits_{i = 2}^{N - 1} c_i\bigg)(0) = 2\bigg(\prod\limits_{i = 2}^{N - 1} c_i\bigg)$. See \Cref{fig:countingcase3}. 

\begin{figure}[h]
    \centering
\centerline{\includegraphics[width=1.1\textwidth]{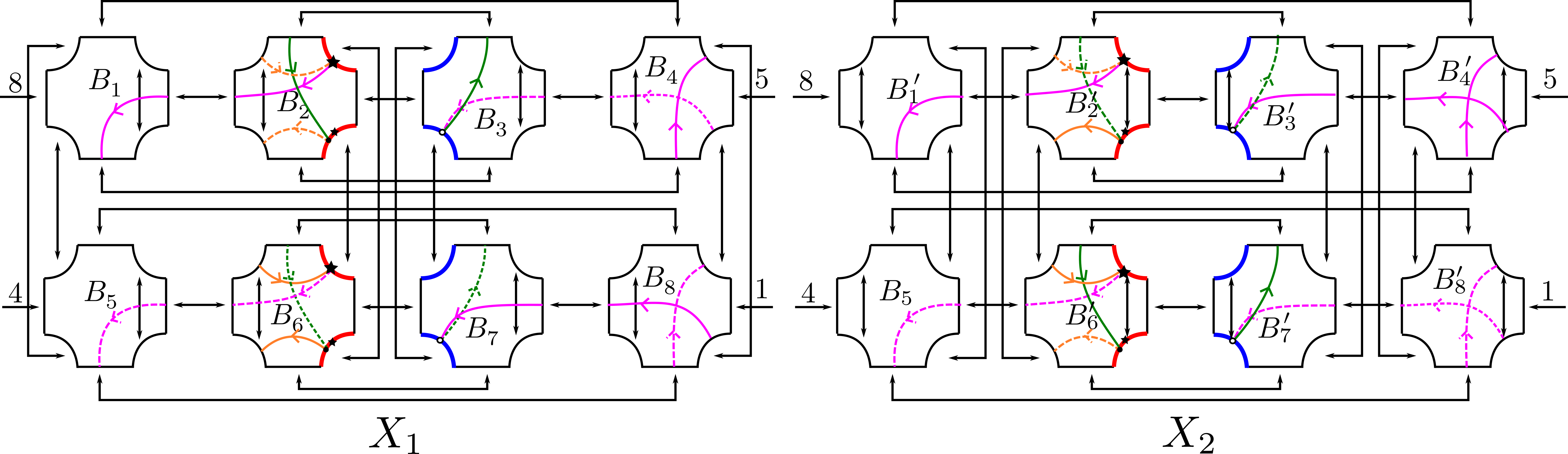}}
    \caption{An illustration for Case 3. Again, $N = 3$ and $c_2 = 2$, so $C_1(\gamma) = 2(2) = C_2(\gamma)$. We can check that there are indeed four identical, non-backtracking closed geodesics in each of $X_1$ and $X_2$.}
    \label{fig:countingcase3}
\end{figure}

We now construct an identical transplanted copy of each $\beta_i^j$, ${\beta'}_i^j = \bigcup\limits_{k = 0}^L {\gamma'}^j_{i, k}$, where ${\gamma'}^j_{i, k} \subset B'_{{n'}^j_{i, k}}$ and ${\beta'}_i^j$ is initiated with the rule $\delta^j_i(0) := {n'}^j_{i, 0} - {n}^j_{i, 0} = 0$ (see \Cref{eqn:delta}). We claim that $\beta_i^j$ meets the gluing curve if and only if ${\beta'}_i^j$ intersects a gluing curve. By \Cref{lemma:delta}, if $\delta^j_i(k) = 0$ or $4$, crossing a side of length $c$ or $b$ does not change $\delta^j_i$ and crossing a side of length $a$ changes $\delta^j_i$ to $\delta^j_i + 4$. Thus, $\delta^j_i(k) = 0$ or $4$ for all $0 \leq k \leq L$; that is, $n^j_{i,k}$ and ${n'}^j_{i, k}$ differ by either $0$ or $4$ (see, for example, any of the previous 3 figures). Then $\beta_i^j$ meets a gluing curve if and only if ${\beta'}_i^j$ does. This establishes a natural bijection between $\{\beta^j_i\}_{j = 1}^2$ and $\{{\beta'}^j_i\}_{j = 1}^2$, so one can make the same computations as before for calculating $C_2(\gamma)$, which counts the number of identical copies of $\gamma$ in $X_2$. We obtain in the end that $C_1(\gamma) = C_2(\gamma)$. One can therefore bijectively map copies of $\gamma$ in $X_1$ to those in $X_2$. 
\end{proof} 

\subsection{Non-homeomorphic, iso-length-spectral surface amalgams}

We first remind the reader of a convenient way to determine whether two simple, thick surface amalgams are homeomorphic, using a criterion established by Lafont in \cite{lafont}. 

\begin{proposition}[Corollary 3.4 of \cite{lafont}]
\label{prop:homeoSA} If $f: \widetilde{X_1} \rightarrow \widetilde{X_2}$ is a quasi-isometry, then $f$ induces a bijection between homeomorphic chambers of $X_1$ and $X_2$. 
\end{proposition}

Recall that if two compact metric spaces are homeomorphic, then there is a quasi-isometry between their universal covers by the Milnor-Schwarz Lemma. From \Cref{prop:homeoSA}, it then follows that the chambers of two homeomorphic simple, thick surface amalgams are necessarily in bijective correspondence with each other. With this in mind, we now construct two isopectral, non-homeomorphic hyperbolic surface amalgams. 

\begin{construction}\label{nonhomeo}
Consider $S_1$ and $S_2$ from \cite{buser1} which have sets of systoles $\{\gamma_i\}_{i = 1}^4$ and $\{\gamma'_i\}_{i = 1}^4$. Construct $X_1$ by identifying all the systoles $\gamma_i \subset X_1$ according to the orientations specified in \Cref{fig:conncomps}. Similarly, glue together all the $\gamma'_i \subset S_2$ to construct $X_2$.  
\end{construction}

\begin{proposition}\label{prop:nonhomeo} The surface amalgams $X_1$ and $X_2$ from Construction \ref{nonhomeo} are iso-length-spectral but not homeomorphic. 
\end{proposition}

\begin{proof} \noindent \textbf{$X_1$ and $X_2$ are not homeomorphic.}
Cutting along the gluing curves yields one chamber for $X_2$ and two for $X_1$, so by \Cref{prop:homeoSA} it is impossible to establish a bijection between the two collections of chambers (see \Cref{fig:nonhomeo} and \Cref{fig:conncomps}). \\

\begin{figure}[h!]
    \centering
\includegraphics[width=0.7\textwidth]{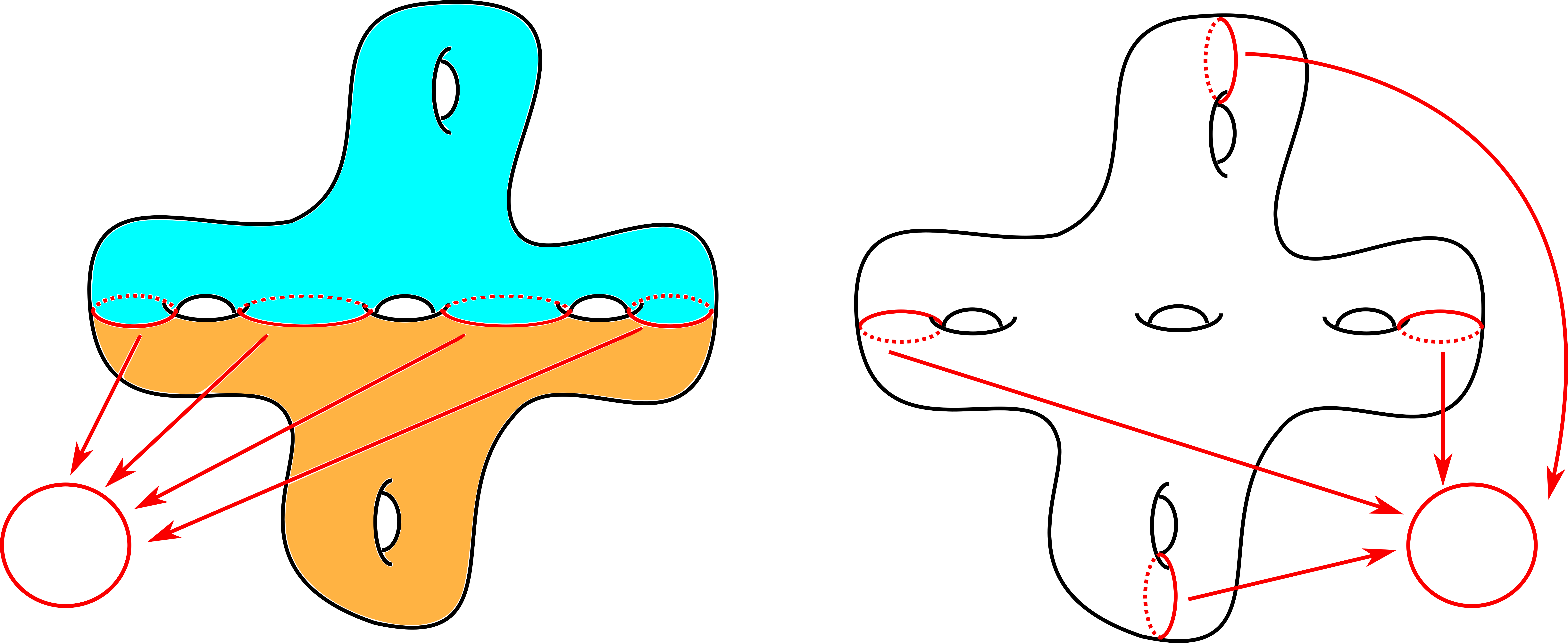}
    \caption{$X_1$ and $X_2$ are homeomorphic to the surface amalgams on the left and right respectively.}
    \label{fig:nonhomeo}
\end{figure}

\begin{figure}[h!]
    \centering
\centerline{\includegraphics[width=1.1\textwidth]{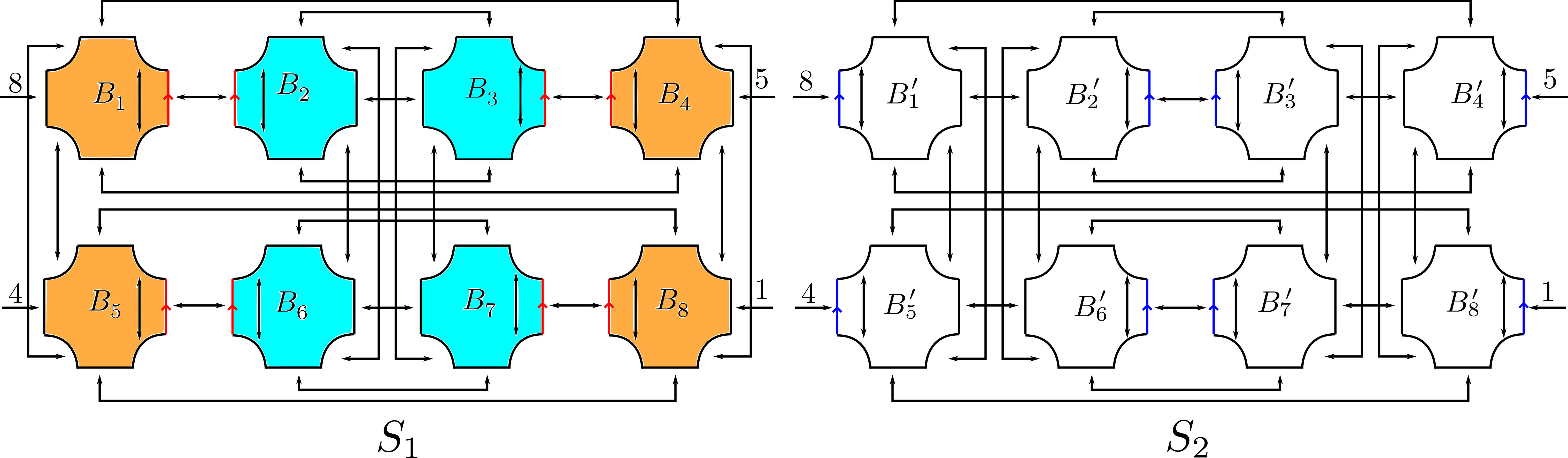}}
    \caption{The two chambers of $X_1$ are in orange and blue while $X_2$ only has one chamber.}
    \label{fig:conncomps}
\end{figure}

\noindent \textbf{$X_1$ and $X_2$ are iso-length-spectral.} We follow a similar strategy as before, by showing that the number of closed geodesics identical to a given closed geodesic $\gamma \subset X_1$ is the same in $X_1$ and $X_2$. Again, we focus our attention on closed geodesics that do not intersect the gluing curve. We decompose $\gamma$ into a union $\bigcup\limits_{i = 1}^N \beta_i$ of geodesic segments that project to connected geodesic segments in $S_1$ beginning and ending on the gluing curve and count the number of ways to concatenate admissible translates of each $\beta_i$.

We will show that in $X_1$, there are four translates of $\beta_i$ with the same beginning and end points on the gluing curve. If $\beta_i$ begins on the right (resp. left) edge of a building block, we show there is one copy starting on the right edge (resp. left) of each odd (resp. even) labeled building block. In particular, we claim that each translated copy of $\beta_i$ will end on a building block with the same parity as the one $\beta_i$ ends on (and thus share an endpoint with $\beta_i$). 

Let $\{\beta^j_i\}_{j = 1}^4$ be the set of geodesic segments in $X_1$ which are translated copies of $\beta_i$ that either all begin on the left sides of the even labeled building blocks or all begin on the right sides of the odd labeled building blocks, depending on where $\beta_i$ begins. Suppose $\beta_i^1 := \beta_i$. We check that $\beta_i^j$, where $j = 2, 3,$ or $4$, ends on a building block whose label has the same parity as the one $\beta_i^1$ ends on, allowing it to be concatenated with the next geodesic segment. Let $\beta_i^j = \bigcup\limits_{k = 0}^{L} \gamma^j_{i, k}$, where each $\gamma^j_k$ is contained in a single building block, $B_{n^j_k}$. Define $d^j_i(k) := n_{i, k}^j - n_{i, k}^1$. Note that $d^j_i(0)$ is even since each $\beta_i^j$ must start on either an odd or even labeled building block. Crossing a side of length $c$ will not change $d^j_i(k)$. Furthermore, crossing a side of length $b$ (resp. $a$) will change (resp. not change) the parity of both $n_{i, k}^j$ and $n_{i, k}^1$, so the parity of $d^j_i(k)$ remains the same for all $0 \leq k \leq L$. In particular, $d^j_i(L)$ is even, so the parity of $n_{i, L}^j$ will match that of $n_{i, 1}^j$. 

We can now compute $C_1(\gamma)$ using the same strategy as before. Using the same argument from \Cref{subsection:homeo} (replacing 1 and 2 with 3 and 4 respectively), we deduce that for each $2 \leq i \leq N - 1$: 
\begin{align*} c_i &:= \#\{\beta^j_i \text{ which can be concatenated with some fixed copy of } \beta_{i - 1}\} \\&= \begin{cases}
    3 \text{ if $\beta_{i - 1}$ ends on a right (resp. left) edge, and $\beta_i$ begins on a right (resp. left) edge;}\\
    4 \text{ otherwise.}
\end{cases}
\end{align*}

We then calculate $c_N$ depending on whether the following facts are true: 

\begin{itemize} 
\item \textit{Fact 1:} $\beta_{N - 1}$ ends on a right edge and $\beta_N$ begins on a right edge (or the statement is true if we replace ``right" with ``left"); 
\item \textit{Fact 2:} $\beta_N$ ends on a right edge and $\beta_1$ begins on a right edge (or the statement is true if we replace ``right" with ``left").
\end{itemize} 

By the same proof as before but replacing $0, 1$ and $2$ with $2, 3$, and $4$ respectively, we have $C_1(\gamma) = 4\bigg(\prod\limits_{i = 2}^{N - 1} c_i\bigg)(4)$ for Case 1, $C_1(\gamma) = 4\bigg(\prod\limits_{i = 2}^{N - 1} c_i\bigg)(3)$ for Case 2, and $C_1(\gamma) = \dfrac{4}{16}(4)\bigg(\prod\limits_{i = 2}^{N - 1} c_i\bigg)(3) + \dfrac{12}{16}(4)\bigg(\prod\limits_{i = 2}^{N - 1} c_i\bigg)(2) = \bigg(\prod\limits_{i = 2}^{N - 1} c_i\bigg)(3) + 3\bigg(\prod\limits_{i = 2}^{N - 1} c_i\bigg)(2)$ for Case 3. The fractions from Case 3 come from the fact that there are now a total of 16 choices of pairs of geodesic segments $\beta_1^j$ and $\beta_{N - 1}^k$ where $j, k \in \{1, 2, 3, 4\}$. For exactly four of these pairs, there is exactly one inadmissible copy of $\beta_N$ which begins on the same building block as the end of $\beta_{N - 1}^k$ and ends on the same building block as the beginning of $\beta_1^j$. For the 12 other pairs, there are two inadmissible copies of $\beta_N$. 

We now show $C_2(\gamma) = C_1(\gamma)$. First, we claim there are four transplanted copies of $\beta_i$ that begin on the right (resp. left) side of each odd (resp. even) labeled building block and end at the same point. There is a natural isometry between gluing geodesics in $X_1$ and those in $X_2$; if a gluing geodesic $c_n$ is the right edge of $B_n$ and the left edge of $B_{n + 1}$, let $\varphi(c_n)$ be the right edge of $B_{n + 1}$ and the left edge of $B_{n + 2}$. If each $\beta_i^j$ begins at $x \in X_1$, we construct a set $\{{\beta'}_i^j\}_{j = 1}^4 = \bigg\{\bigcup\limits_{k = 0}^{L} {\gamma'}^j_{i, k} \bigg\}_{j = 1}^4$ of transplanted copies of $\beta_i^j$ which begin at $\varphi(x) \in X_2$ and pass through building blocks $B'_{{n'}_{i, k}^j}$. Initiate each ${\beta'}_{i, k}^j$ with the rule $\delta^j_{i}(0) := {n'}_{i, 0}^j - n_{i, 0}^j = 1$. By \Cref{lemma:delta}, if $\delta^j_i(k) = \pm 1$, crossing sides of length $a$ or $c$ does not affect $\delta^j_i(k)$ while crossing a side of length $b$ sends $\pm 1$ to $\mp 1$. Regardless of the sequences of edges crossed, ${\delta'}^j_i(L) = 1$ or $-1$ for all $j$. Thus, each ${\beta'}_i^j$ will end on the same side of a gluing geodesic that each $\beta_i^j$ ends on. Each ${\beta'}_i^j$ will also end on a distinct building block. 

Thus, as before, one can make the same computations for $C_2(\gamma)$ that counts the number of identical copies of $\gamma$ in $X_2$. We obtain in the end that $C_1(\gamma) = C_2(\gamma)$. As a result, one can bijectively map copies of $\gamma$ in $X_1$ to copies of $\gamma$ in $X_2$. 
\end{proof}

\begin{remark} We remark that $X_1$ and $X_2$ are commensurable. Indeed, they have a common double cover. The chambers in $X_1$, which are both tori with four boundary components, lift to their double covers, which are tori with 8 boundary components, $S_{1, 8}$. In other words, the chambers in $X_1$ each lift to a copy of the chamber in $X_2$. In summary, $X_1$ and $X_2$ both have double covers consisting of two copies of $S_{1, 8}$ and two gluing curves obtained from identifying two quartets of lifts of boundary components originally identified in $X_1$ and $X_2$. 
\end{remark}

\section{Proof of \Cref{theorem:commensurable}}

We now prove \Cref{theorem:commensurable}, our second main result following the construction below. 

\begin{construction} Consider, again, the surfaces from \cite{buser1}. We will glue three copies of each surface together to create $X_1$ and $X_2$. For both $X_1$ and $X_2$, there will be one gluing curve which will consist of unions of perpendiculars between edges of length $b$ that bisect all the building blocks except ones labeled $6$ and $7$. 

Note that in each surface, three pairs of perpendiculars form geodesics: the ones in building blocks with labels $1$ and $4$, $2$ and $3$, and $5$ and $8$. As a result, we need to specify how the geodesics are glued together. We will identify the halves of geodesics on the odd numbered building blocks and halves of geodesics on even numbered building blocks. See \Cref{fig:not1-1} for the orientations of the gluing curves.
\end{construction} 

\noindent \textbf{$X_1$ and $X_2$ are weak length isospectral.} We use Buser's transplantation technique. As before, it suffices to consider closed geodesics that are not closed geodesics in copies of $S_1$ and $S_2$. We decompose $\gamma \subset X_1$ into a union of geodesic segments $\bigcup\limits_{j = 1}^N \gamma_j$ each of which is contained entirely within a building block $B_{n_j}$. We must specify an algorithm to construct a transplanted copy of $\gamma$, $\gamma' \subset X_2$, which is a union $\bigcup\limits_{k = 1}^N \gamma'_k$ of geodesic segments that are each completely contained within a building block in $S_2$, start and end on the gluing curve in $X_2$, and do not backtrack. We will label the copies of $S_1$ in $X_1$ with $S_1^1$, $S_1^2$, and $S_1^3$ and similarly label the copies of $S_2$ in $X_2$. In the following algorithm, we will write $\delta^s(j) = {n'}^s_j - n^s_j$, which gives instructions for transplanting geodesic segments from $S_1^s$ to $S_2^s$. 

\begin{algorithm} 
\label{algorithm:noncommensurable}
Suppose $\gamma_{j - 1} \cup \gamma_j$ does not project to a connected geodesic segment in any copy of $S_1$. Moreover, suppose that following \Cref{algorithm:originit}, $\beta_{j, j + L} := \bigcup\limits_{i = 0}^L \gamma_{j + i}$ projects to a continuous geodesic segment in some copy of $S_1$, $S_1^s$, where $1 \leq s \leq 3$. We then initiate $\gamma'_j$ using the following rules (unless otherwise specified, $\gamma'_j$ will be initiated on $S_2^s$): 
\begin{enumerate}
%\item \textit{Case 1: $\# a$ is even}: Initiate $\gamma'_j$ on $S_2^s$ using $\delta^s(j) = 0$; 
%\textit{Case 2: $\# a$ is odd}: 
\item If $\gamma_j$ lies in a building block indexed by an element in the set $\{1, 4, 5, 8\}$, we set $\delta^s(j) = 0$ if $\# b$ is even and $4$ if $\# b$ is odd. 
\item If $\gamma_{j + L}$ lies in a building block indexed by an element in the set $\{1, 4, 5, 8\}$, we set $\delta^s(j) = 0$. 
\item Otherwise, set $\delta^s(j) = 2$. 

%Otherwise, set $\delta^s(j) = 1$, unless $\gamma_j$ is in $B_5$, ($\# b$ is even and $n_{j + L} = 5$), or ($\# b$ is odd and $n_{j + L} = 8$). For the remaining three cases, set $\delta^s(j) = 4$ if $n_{j + L} = 5$ and $\delta^s(j) = 0$ otherwise.
\item If any of the previous rules cause backtracking, copy $\gamma'_j$ over to another copy of $B'_{n_j}$ (e.g. ${B'}^t_{n_j} \subset S_2^t$ where $s \neq t$). 
\end{enumerate} 
\end{algorithm} 

\begin{remark}
    Three surfaces are needed in order to ensure (4) from \Cref{algorithm:noncommensurable} can always be applied in order to prevent backtracking. In fact, only two surfaces are needed to correct for backtracking when constructing $\gamma'_j$, where $j < N$, but for $j = N$, one needs to ensure that $\gamma'_N$ is admissible with respect to both $\gamma'_1$ and $\gamma'_{N - 1}$. 
\end{remark}

We now show that following (1), (2) and (3) from \Cref{algorithm:noncommensurable}, for each $\beta_{j, j + L} \subset S_1^s$, we can obtain a geodesic segment $\beta'_{j, j + L} \in S_2^s$ which begins and ends on the gluing curve. If $\gamma_{j} \subset B_{n_j}$ for some $n_j \in \{1, 4, 5, 8\}$, then we set $\delta^s(j) = 0$ if $\# b$ is even and $\delta^s(j) = 4$ if $\# b$ is odd so that regardless, $\gamma'_j \subset B'_{n'_j}$, where $n'_j = k + 4 \in \{1, 4, 5, 8\}$. By \Cref{lemma:delta}, if $\delta^s = 0$ or $4$, crossing a side of length $a$ or $c$ leaves $\delta^s$ invariant, while crossing a side of length $b$ replaces $\delta^s$ by $\delta + 4$. Thus, $\delta^s(j + L) = 0$, so $B^s_{{n'}_{(j + L)}}$ intersects a gluing curve since $B^s_{n_{(j + L)}}$ does. Suppose $\gamma_{j + L} \subset B_k$ where $k \in \{1, 4, 5, 8\}$. Then setting $\delta^s(j) = 0$ ensures that $\gamma'_j$ begins on a building block intersecting a gluing curve and that $\delta^s(j + L)$ is $0$ or $4$. Thus, $\gamma'_{j + L}$ lies in a building block indexed by an element of the set $\{1, 4, 5, 8\}$, which necessarily intersects the gluing curve. 

This leaves the case ($\gamma_j \subset B_{n_j}$, where $n_j \in \{2, 3\}$) \textit{and} ($\gamma_{j + L} \subset B_{n_{j + L}}$, where $n_{j + L} \in \{2, 3\}$). We set $\delta^s(j) = 2$. Since $n_j \in \{2, 3\}$, $n'_j \in \{1, 4, 5, 8\}$ and the building blocks indexed by this set all intersect the gluing curve. By \Cref{lemma:delta}, $\delta = \pm 2$, crossing a side of length $a$ or $b$ replaces $\delta$ by $\delta + 4$, while as before, crossing a side of length $c$ does not change $\delta$. Thus, if $\# a + \#b$ is even (resp. odd), then $\delta^s(j + L) = 2$ (resp. $-2$). Either way, since $n^s_{j + L} \in \{2, 3\}$, we have that ${n'}^s_{j + L} \in \{1, 4, 5, 8\}$, and the building blocks indexed by this set all intersect the gluing curve. 

We need to check an additional condition to ensure $\gamma'_j \cup \gamma'_{j - 1}$ is connected in $X_2$. It suffices to show if $\gamma'_{j - 1}$ ends in an odd (resp. even) numbered building block, $\gamma'_j$ also begins in an odd (resp. even) numbered building block. We know the previous sentence to be true if $\gamma'$ is replaced with $\gamma$. In all cases, $\delta^s(k)$ is even for any $j \leq k \leq j + L$. As a consequence, the parities of the building blocks containing $\gamma^s_j$ and ${\gamma'}^s_j$ are the same. This is enough to prove what we want.

\begin{figure}[h]
    \centering
    \centerline{\includegraphics[width=1.1\textwidth]{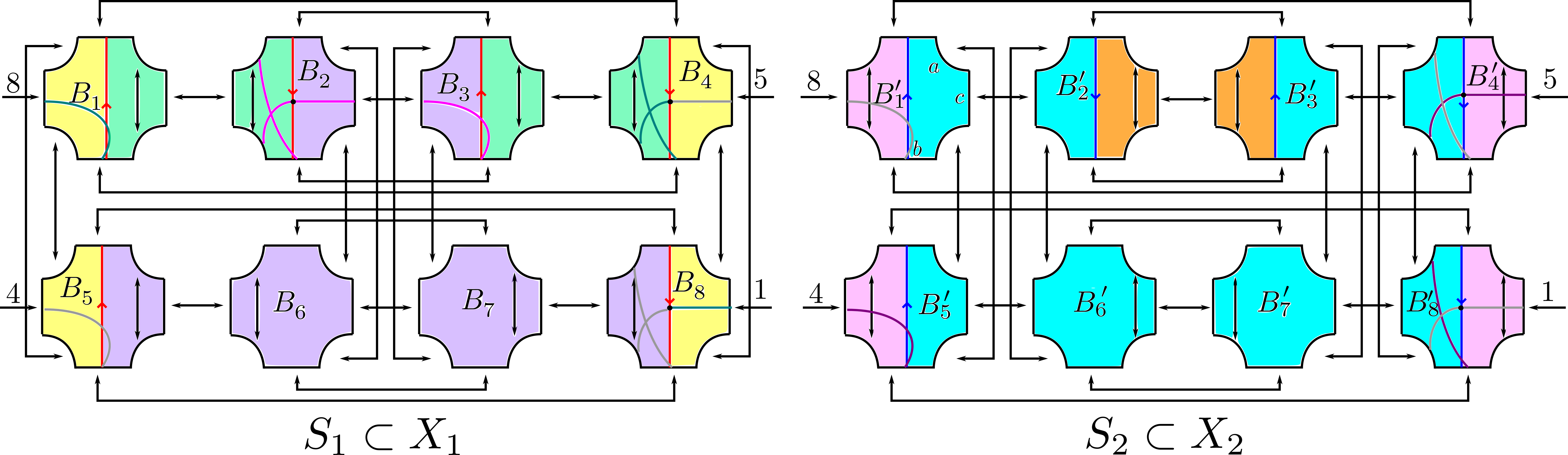}}
    \caption{Weak length isospectral but noncommensurable $X_1$ and $X_2$ with the chambers shown in different colors. The gluing curves in $X_1$ are shown in red, while the ones in $X_2$ are dark blue. The pink and green geodesics in $X_1$ both map to the purple geodesic in $X_2$, so \Cref{algorithm:noncommensurable} does not give a 1-1 correspondence between transplanted closed geodesics.}
    \label{fig:not1-1}
\end{figure}

\begin{remark}
    Unfortunately, \Cref{algorithm:noncommensurable} does not yield a 1-1 correspondence between identical geodesic segments in $X_1$ and $X_2$, even if we restrict to single subsets of $X_1$ and $X_2$ consisting of a single copy of $S_1$ or $S_2$ with geodesics identified. \Cref{fig:not1-1} illustrates this. Following \Cref{algorithm:noncommensurable}, both the pink and green curve in $X_1$ are assigned to the purple curve in $X_2$. Thus, one cannot show iso-length-spectrality using transplantation via \Cref{algorithm:noncommensurable}. In the picture, we also show all the identical closed geodesics that begin at the same point (indicated by black dots) in $X_1$ and $X_2$ respectively. Note that there are three closed geodesics in $X_1$ and only two in $X_2$, as initiating a geodesic from $B'_2$ does not result in a closed geodesic in $X_2$.
\end{remark}

 One may ask whether we can come up with initiation rules so that there is indeed a 1-1 correspondence between identical closed geodesics. Unfortunately, the answer is no. We show this by counting identical copies of a particular closed geodesic in $X_1$ and $X_2$.  

 We revisit the closed geodesics in \Cref{fig:not1-1}. In $X_1$, notice that on each surface $S_1^s$, there are three identical closed geodesics (depicted in green, pink, and gray) that begin and end at the same point $x \in X_1$, which is depicted as a black dot. In contrast, on each surface $S_2^s$ in $X_2$, there are only two closed geodesics beginning and ending at the corresponding point $x \in X_2$. In $S_1^2 \subset X_1$, there are also two geodesics beginning and ending in the odd indexed building blocks, on the half of the gluing $x$ is not on. In $S_2^s \subset X_2$, there are also two such closed geodesics; see \Cref{fig:not1-12}. Thus, in total, since there are $3$ copies of each surface, there are $5(3) = 15$ copies of identical closed geodesics in $X_1$ but only $4(3) = 12$ copies of the same closed geodesics in $X_2$. This shows it is not possible to achieve a 1-1 correspondence between identical closed geodesics in $X_1$ and $X_2$; thus, one cannot prove iso-length-spectrality via transplantation. We remark that this does not rule out iso-length-spectrality of $X_1$ and $X_2$, which is improbable but still possible.\\

\begin{figure}[H]    \centerline{\includegraphics[width=1.1\textwidth]{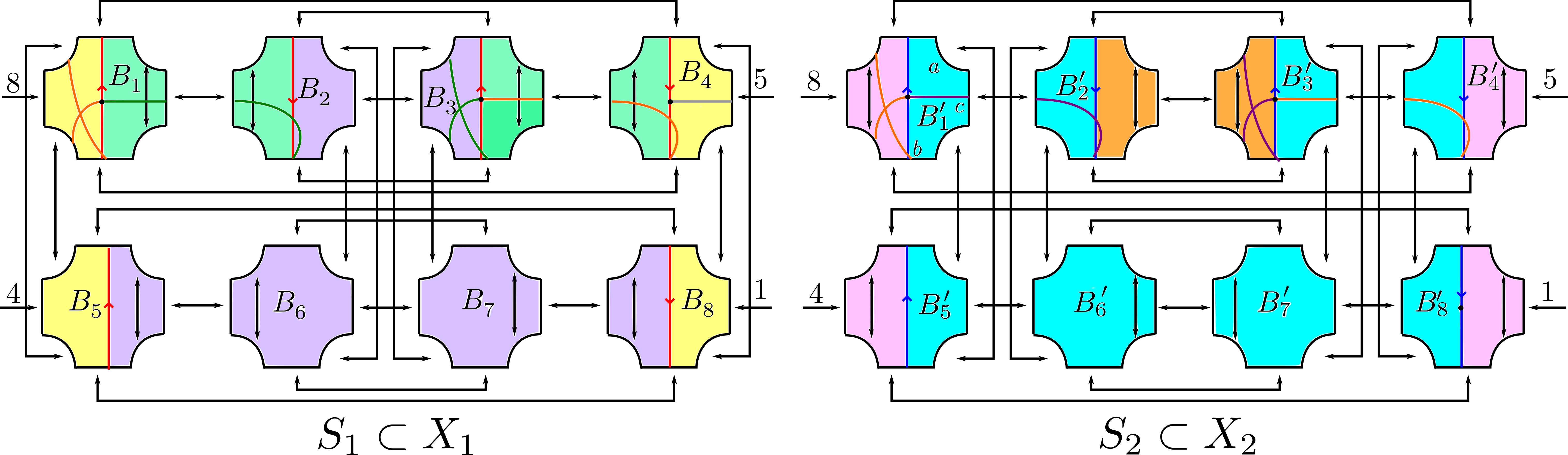}}
    \caption{Copies of closed geodesics identical to the ones shown in \Cref{fig:not1-1} which begin and end in the odd-labeled building blocks.}
    \label{fig:not1-12}
\end{figure}

\begin{figure}[h!]
    \centering
    \includegraphics[width=0.8\textwidth]{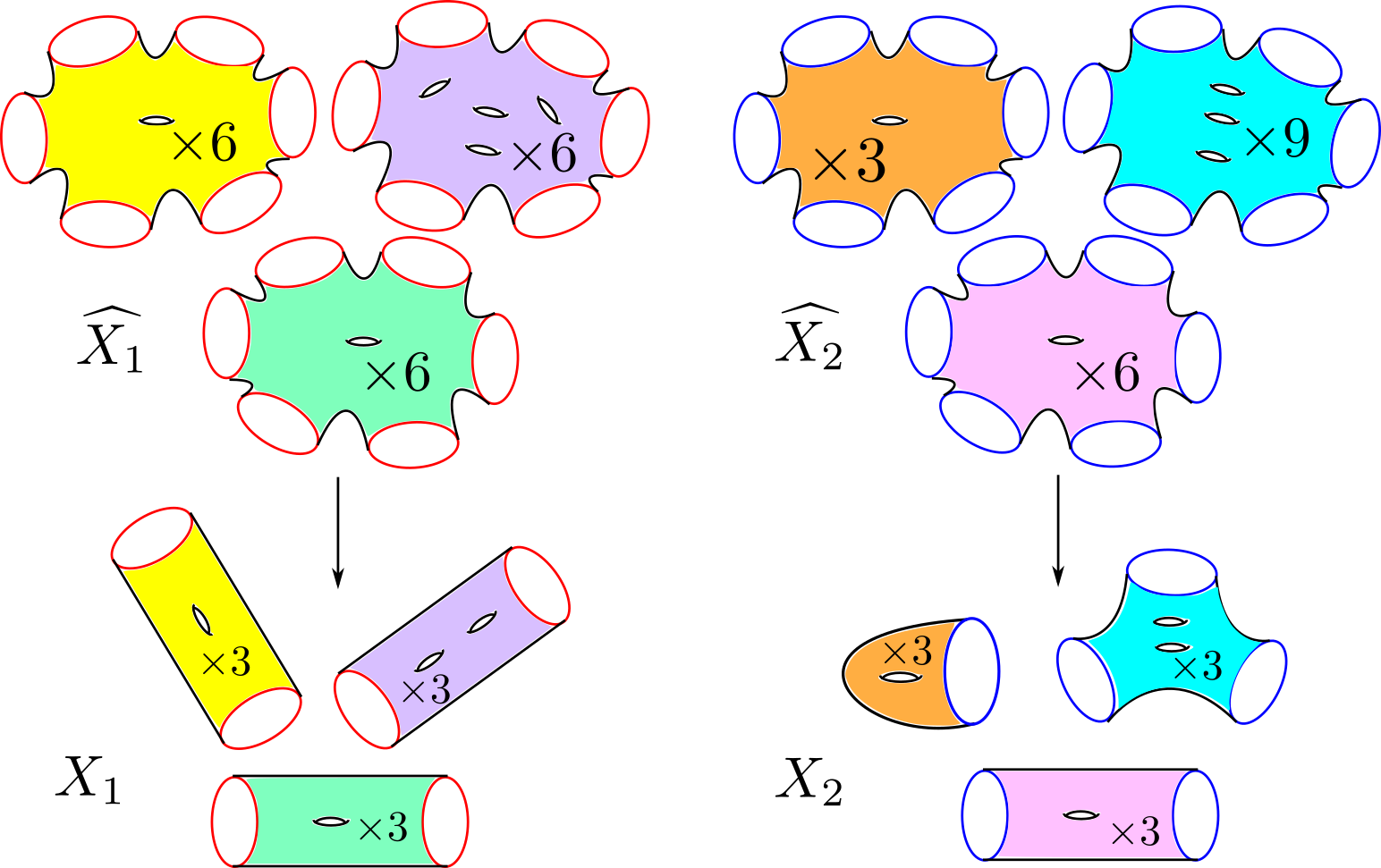}
    \caption{$X_1$ and $X_2$, which are constructed by gluing together the red and blue sets of boundary components, are not commensurable. Their six-sheeted covers referenced in the proof of \Cref{theorem:commensurable} are also shown.}
    \label{fig:noncomm}
\end{figure}

\noindent 
\textbf{$X_1$ and $X_2$ are not commensurable.} We now show that $X_1$ and $X_2$ do not have a common finite-sheeted cover. Note that $X_1$ has six tori with two boundary components, which we denote by $S_{1, 2}$, and three genus two surfaces with two boundary components $S_{2, 2}$; again, all the boundary components are identified. On the other hand, $X_2$ consists of three tori with one boundary component $S_{1, 1}$, three tori with two boundary components $S_{1, 2}$, and three genus two surfaces with three boundary components $S_{2, 3}$, and all the boundary components are identified together. 

Consider a six-sheeted cover of $X_1$, $\widehat{X_1}$, which consists of six copies of a genus four surface with six boundary components (three-sheeted covers of $S_{2, 2}$) and 12 copies of a torus with six boundary components (three-sheeted covers of $S_{1, 2}$). Consider also six-sheeted cover of $X_2$, $\widehat{X_2}$, which consists of nine copies of genus three surfaces with six boundary components (double covers of $S_{2, 3}$) and nine copies of a torus with six boundary components (six three-sheeted covers of $S_{1, 2}$ and three six-sheeted covers of $S_{1, 1}$) (see \Cref{fig:noncomm}). For both $\widehat{X_1}$ and $\widehat{X_2}$, there are six gluing curves, and exactly one boundary component in each chamber is glued to each gluing curve. Note that the Euler characteristics of the chambers in $\widehat{X_1}$, in ascending order, are $\{\underbrace{-12,...,-12}_{\times 6}, \underbrace{-6, ..., -6}_{\times 12}\}$ while those in $\widehat{X_2}$ are $\{\underbrace{-10, ..., -10}_{\times 9}, \underbrace{-6, ..., -6}_{\times 9}\}$. Let $\{S_i\}$ and $\{T_i\}$ denote the collections of chambers of $\widehat{X_1}$ and $\widehat{X_2}$ respectively, labeled so that $\chi(S_1) \leq \chi(S_2) \leq... \leq \chi(S_{18})$ and $\chi(T_1) \leq \chi(T_2) \leq... \leq \chi(T_{18})$. A generalization of Proposition 3.3.2 of \cite{stark}, proved in Section 5.2 of \cite{DST}, implies that $\pi_1(\widehat{X_1})$ and $\pi_1(\widehat{X_2})$ are abstractly commensurable (e.g. they do not have isomorphic finite-index subgroups) if and only if all the $\dfrac{\chi(S_i)}{\chi(T_i)}$ are equal. Note, however, that $\dfrac{\chi(S_1)}{\chi(T_1)} = \dfrac{-12}{-10} \neq \dfrac{-6}{-6} = \dfrac{\chi(S_{18})}{\chi(T_{18})}$. It then follows that $\pi_1(\widehat{X_1})$ and $\pi_1(\widehat{X_2})$ are not abstractly commensurable. 

Since abstract commensurability is an equivalence relation, $\pi_1(X_1)$ and $\pi_1(X_2)$ also cannot be abstractly commensurable since $\pi_1(X_i)$ is abstractly commensurable to $\pi_1(\widehat{X_i})$ for $i = 1, 2$. By the Galois correspondence of covering spaces for CW complexes, it then follows that $X_1$ and $X_2$ also cannot be commensurable, as claimed. \qed

\bibliographystyle{amsalpha}
\bibliography{biblio}

\providecommand{\bysame}{\leavevmode\hbox to3em{\hrulefill}\thinspace}
\providecommand{\MR}{\relax\ifhmode\unskip\space\fi MR }
\providecommand{\MRhref}[2]{%
  \href{http://www.ams.org/mathscinet-getitem?mr=#1}{#2}
}
\providecommand{\href}[2]{#2}
\begin{thebibliography}{LMNR07}

\bibitem[Bus86]{buser1}
Peter Buser, \emph{Isospectral {R}iemann surfaces}, Ann. Inst. Fourier (Grenoble) \textbf{36} (1986), no.~2, 167--192.

\bibitem[Bus92]{buser2}
\bysame, \emph{Geometry and spectra of compact {R}iemann surfaces}, Progress in Mathematics, vol. 106, Birkh\"{a}user Boston, Inc., Boston, MA, 1992.

\bibitem[Cha84]{chavel}
Isaac Chavel, \emph{Eigenvalues in {R}iemannian geometry}, Pure and Applied Mathematics, vol. 115, Academic Press, Inc., Orlando, FL, 1984, Including a chapter by Burton Randol, With an appendix by Jozef Dodziuk.

\bibitem[CHLR08]{chlr}
T.~Chinburg, E.~Hamilton, D.~D. Long, and A.~W. Reid, \emph{Geodesics and commensurability classes of arithmetic hyperbolic 3-manifolds}, Duke Math. J. \textbf{145} (2008), no.~1, 25--44.

\bibitem[DST18]{DST}
Pallavi Dani, Emily Stark, and Anne Thomas, \emph{Commensurability for certain right-angled {C}oxeter groups and geometric amalgams of free groups}, Groups Geom. Dyn. \textbf{12} (2018), no.~4, 1273--1341.

\bibitem[Ger70]{gerst}
I.~Gerst, \emph{On the theory of {$n{\rm th}$} power residues and a conjecture of {K}ronecker}, Acta Arith. \textbf{17} (1970), 121--139.

\bibitem[HST20]{hst}
G.~Christopher Hruska, Emily Stark, and Hung~Cong Tran, \emph{Surface group amalgams that (don't) act on 3-manifolds}, Amer. J. Math. \textbf{142} (2020), no.~3, 885--921.

\bibitem[Laf07]{lafont}
Jean-Fran\c{c}ois Lafont, \emph{Diagram rigidity for geometric amalgamations of free groups}, J. Pure Appl. Algebra \textbf{209} (2007), no.~3, 771--780.

\bibitem[LMNR07]{lmnr}
C.~J. Leininger, D.~B. McReynolds, W.~D. Neumann, and A.~W. Reid, \emph{Length and eigenvalue equivalence}, Int. Math. Res. Not. IMRN (2007), no.~24, Art. ID rnm135, 24.

\bibitem[LSV06]{lsv}
Alexander Lubotzky, Beth Samuels, and Uzi Vishne, \emph{Isospectral {C}ayley graphs of some finite simple groups}, Duke Math. J. \textbf{135} (2006), no.~2, 381--393. \MR{2267288}

\bibitem[McR14]{mcreynolds}
D.~B. McReynolds, \emph{Isospectral locally symmetric manifolds}, Indiana Univ. Math. J. \textbf{63} (2014), no.~2, 533--549.

\bibitem[Mil64]{milnor}
J.~Milnor, \emph{Eigenvalues of the {L}aplace operator on certain manifolds}, Proc. Nat. Acad. Sci. U.S.A. \textbf{51} (1964), 542.

\bibitem[PR15]{GR}
Gopal Prasad and Andrei~S. Rapinchuk, \emph{Weakly commensurable groups, with applications to differential geometry}, Handbook of group actions. {V}ol. {I}, Adv. Lect. Math. (ALM), vol.~31, Int. Press, Somerville, MA, 2015, pp.~495--524.

\bibitem[Rei92]{reid}
Alan~W. Reid, \emph{Isospectrality and commensurability of arithmetic hyperbolic {$2$}- and {$3$}-manifolds}, Duke Math. J. \textbf{65} (1992), no.~2, 215--228.

\bibitem[Spa89]{spatzier}
R.~J. Spatzier, \emph{On isospectral locally symmetric spaces and a theorem of von {N}eumann}, Duke Math. J. \textbf{59} (1989), no.~1, 289--294.

\bibitem[Sta17]{stark}
Emily Stark, \emph{Abstract commensurability and quasi-isometry classification of hyperbolic surface group amalgams}, Geom. Dedicata \textbf{186} (2017), 39--74.

\bibitem[Sun85]{sunada}
Toshikazu Sunada, \emph{Riemannian coverings and isospectral manifolds}, Ann. of Math. (2) \textbf{121} (1985), no.~1, 169--186.

\bibitem[Vig80]{vigneras}
Marie-France Vign\'{e}ras, \emph{Vari\'{e}t\'{e}s riemanniennes isospectrales et non isom\'{e}triques}, Ann. of Math. (2) \textbf{112} (1980), no.~1, 21--32.

\end{thebibliography}

\end{document}